\documentclass[11pt,a4paper,reqno]{amsart}

\usepackage{epsfig}
\usepackage{amsfonts}
\usepackage{amsmath}
\usepackage{amssymb}
\usepackage{amsthm}
\usepackage{color}
\usepackage{dsfont}
\usepackage{enumerate}
\usepackage[T1]{fontenc}
\usepackage[latin1]{inputenc}
\usepackage{ae,aecompl}
\usepackage{graphicx}
\usepackage{nicefrac}
\usepackage[includehead,includefoot,margin=25mm]{geometry}

\newtheorem{theorem}{Theorem}
\newtheorem{corollary}[theorem]{Corollary}
\newtheorem{lemma}[theorem]{Lemma}

{\theoremstyle{remark}
  \newtheorem{remark}[theorem]{Remark}}
{\theoremstyle{definition}

}

\newcommand{\PP}[0]{\ensuremath{\mathds{P}}}

\newcommand{\ZZ}[0]{\ensuremath{\mathds{Z}}}

\newcommand{\AF}[0]{\ensuremath{\mathds{A}}}

\newcommand{\QQ}[0]{\ensuremath{\mathds{Q}}}

\newcommand{\TD}[0]{\ensuremath{\mathbf{T}}}
\newcommand{\KK}[0]{\ensuremath{\mathbf{k}}}
\newcommand{\OO}[0]{\ensuremath{\mathcal{O}}}

\newcommand{\DD}[0]{\ensuremath{\mathfrak{D}}}

\newcommand{\diag}[0]{\ensuremath{\operatorname{diag}}}
\newcommand{\fract}[0]{\ensuremath{\operatorname{Frac}}}

\newcommand{\spec}[0]{\ensuremath{\operatorname{Spec}}}

\newcommand{\Aut}[0]{\ensuremath{\operatorname{Aut}}}
\newcommand{\Lie}[0]{\ensuremath{\operatorname{Lie}}}
\newcommand{\chara}[0]{\ensuremath{\operatorname{Char}}}

\newcommand{\rank}[0]{\ensuremath{\operatorname{rank}}}

\newcommand{\trdeg}[0]{\ensuremath{\operatorname{tr.deg}}}

\begin{document}

\title{Roots of the affine Cremona group}

\author{Alvaro Liendo}

\address{Mathematisches Institut,
         Universit\"at Basel,
         Rheinsprung 21,
         CH-4051 Basel,
         Switzerland.}
\email{alvaro.liendo@unibas.ch}

\date{\today}

\thanks{ \mbox{\hspace{11pt}}{\it 2010 Mathematics Subject
    Classification}: 14R10, 14R20, 14L30,13N15.\\
  \mbox{\hspace{11pt}}{\it Key words}: affine Cremona group, roots of
  an algebraic group, torus actions, locally nilpotent derivation.}

\begin{abstract}
  Let $\KK^{[n]}=\KK[x_1,\ldots,x_n]$ be the polynomial algebra in $n$
  variables and let $\AF^n=\spec\KK^{[n]}$. In this note we show that
  the root vectors of the affine Cremona group $\Aut(\AF^n)$ with
  respect to the diagonal torus are exactly the locally nilpotent
  derivations $\mathbf{x}^\alpha\tfrac{\partial}{\partial x_i}$, where
  $\mathbf{x}^\alpha$ is any monomial not depending on $x_i$. This
  answers a question due to Popov.
\end{abstract}

\maketitle

\section*{Introduction}

Letting $\KK$ be an algebraically closed field of characteristic zero,
we let $\KK^{[n]}=\KK[x_1,\ldots,x_n]$ be the polynomial algebra in $n$
variables, and $\AF^n=\spec \KK^{[n]}$ be the affine space. The
Cremona group $\Aut(\AF^n)$ is the group of automorphisms of $\AF^n$,
or equivalently, the group of $\KK$-automorphisms of $\KK^{[n]}$. We
define $\Aut^*(\AF^n)$ as the subgroup of volume preserving
automorphisms i.e.,
$$\Aut^*(\AF^n)=\left\{\gamma\in \Aut(\AF^n) \mid \det\left(
    \frac{\partial}{\partial
      x_i}\gamma(x_j)\right)_{i,j}=1\right\}\,.$$ %
The groups $\Aut(\AF^n)$ and $\Aut^*(\AF^n)$ are infinite
dimensional algebraic groups \cite{Sha66,Kam79}.

It follows from \cite{Bia66,Bia67} that the maximal dimension of an
algebraic torus contained in $\Aut^*(\AF^n)$ is $n-1$. Moreover, every
algebraic torus of dimension $n-1$ contained in $\Aut^*(\AF^n)$ is
conjugated to the diagonal torus
$$\TD=\left\{\gamma=\diag(\gamma_1,\ldots,\gamma_n)\in \Aut^*(\AF^n)
  \mid \gamma_1\cdots \gamma_n=1\right\}\,.$$

A $\KK$-derivation $\partial$ on an algebra $A$ is called locally
nilpotent (LND for short) if for every $a\in A$ there exists $k\in
\ZZ_{\geq 0}$ such that $\partial^k(a)=0$. If
$\partial:\KK^{[n]}\rightarrow \KK^{[n]}$ is an LND on the polynomial
algebra, then $\exp(t\partial)\in \Aut^*(\AF^n)$, for all $t\in \KK$
\cite{Fre06}. Hence, $\partial$ belongs to the Lie algebra
$\Lie(\Aut^*(\AF^n))$.

In analogy with the notion of root from the theory of algebraic groups
\cite{Spr98}, Popov introduced the following definition.  A non-zero
locally nilpotent derivation $\partial$ on $\KK^{[n]}$ is called a
root vector of $\Aut^*(\AF^n)$ with respect to the diagonal torus
$\TD$ if there exists a character $\chi$ of $\TD$ such that
$$\gamma\circ\partial\circ\gamma^{-1}=\chi(\gamma)\cdot\partial,\quad\mbox{for
  all}\quad \gamma\in\TD\,.$$ %
The character $\chi$ is called the root of $\Aut^*(\AF^n)$ with
respect to $\TD$ corresponding to $\partial$.

Letting $\boldsymbol{\alpha}=(\alpha_1,\ldots,\alpha_n)\in \ZZ_{\geq
  0}^n$, we let $\mathbf{x}^{\boldsymbol{\alpha}}$ be the monomial
$x_1^{\alpha_1}\cdots x_n^{\alpha_n}$. In this note we apply the
results in \cite{Lie10} to prove the following theorem. This answers a
question due to Popov \cite{FrRu05}.
\begin{theorem} \label{theo}
  If $\partial$ is a root vectors of $\Aut^*(\AF^n)$ with respect to
  the diagonal torus $\TD$, then
  $$\partial=\lambda\cdot\mathbf{x}^{\boldsymbol{\alpha}}
  \cdot\frac{\partial}{\partial x_i}\,,$$ for some $\lambda\in \KK^*$,
  some $i\in \{1,\ldots,n\}$, and some
  $\boldsymbol{\alpha}\in\ZZ_{\geq 0}^n$
  such that $\alpha_i=0$. The corresponding root is the character
  $$\chi:\TD\rightarrow \KK^*,\quad
  \gamma=\diag(\gamma_1,\ldots,\gamma_n)\mapsto
  \gamma_i^{-1}\prod_{j=1}^n\gamma_j^{\alpha_j}\,.$$
\end{theorem}

\section{Proof of the Theorem}

It is well known that the set $\chara(\TD)$ of characters of $\TD$
forms a lattice isomorphic to $M=\ZZ^{n-1}$. It is customary to fix an
isomorphism $M\simeq\chara(\TD)$ and to denote the character
corresponding to $m$ by $\chi^m$. The natural $\TD$-action on $\AF^n$
gives rise to an $M$-grading on $\KK^{[n]}$ given by
$$\KK^{[n]}=\bigoplus_{m\in M}B_m,\quad\mbox{where}\quad
B_m=\left\{f\in\KK^{[n]}\mid \gamma(f)=\chi^m(\gamma)f,\, \forall
  \gamma\in \TD\right\}\,.$$

An LND $\partial$ on $\KK^{[n]}$ is called homogeneous if it send
homogeneous elements into homogeneous elements.  Let $f\in
\KK^{[n]}\setminus \ker\partial$ homogeneous. We define the degree of
$\partial$ as $\deg\partial=\deg(\partial(f))-\deg(f)\in M$
\cite[Section 1.2]{Lie10}.
\begin{lemma}\label{root-hom}
  An LND on $\KK^{[n]}$ is a root vectors of $\Aut^*(\AF^n)$ with
  respect to the diagonal torus $\TD$ if and only if $\partial$ is
  homogeneous with respect to the $M$-grading on $\KK^{[n]}$ given by
  $\TD$. Furthermore, the corresponding root is the character
  $\chi^{\deg\partial}$.
\end{lemma}
\begin{proof}
  Let $\partial$ be a root vector of $\Aut^*(\AF^n)$ with root
  $\chi^e$, so that
  $\partial=\chi^{-e}(\gamma)\cdot\gamma\circ\partial\circ\gamma^{-1}$,
  $\forall\gamma\in\TD$. We consider a homogeneous element $f\in
  B_{m'}$ and we let $\partial(f)=\sum_{m\in M} g_m$, where $g_m$ is
  homogeneous, so that
  $$\sum_{m\in M}g_m=\partial(f)=\chi^{-e}(\gamma)\cdot\gamma\circ
  \partial\circ\gamma^{-1}(f)=\chi^{-e-m'}(\gamma)\sum_{m\in
    M}\chi^m(\gamma)\cdot g_m,\quad\forall\gamma\in\TD\,.$$ %
  This equality holds if and only if $g_m=0$ for all but one $m\in
  M$ i.e., if $\partial$ is homogeneous. In this case,
  $\partial(f)=g_m=\chi^{-e-m'+m}(\gamma)\cdot \partial(f)$, and so
  $e=m-m'=\deg(\partial(f))-\deg(f)=\deg\partial$.
\end{proof}

In \cite{AlHa06}, a combinatorial description of a normal affine
$M$-graded domain $A$ is given in terms of polyhedral divisors, and in
\cite{Lie10} a description of the homogeneous LNDs on $A$ is given in
terms of these combinatorial data in the case where $\trdeg A=\rank
M+1$. In the following we apply these results to compute the
homogeneous LNDs on the $M$-graded algebra $\KK^{[n]}$. First, we give
a short presentation of the combinatorial description in \cite{AlHa06}
in the case where $\trdeg A=\rank M+1$. For a more detailed treatment
see \cite[Section 1.1]{Lie10}.

Let $N$ be the dual lattice of $M$. The combinatorial description in
\cite{AlHa06} deals with the following data: A pointed polyhedral cone
$\sigma\in N_\QQ:=N\otimes \QQ$ dual to the weight cone
$\sigma^\vee\subseteq M_\QQ:=M\otimes\QQ$ of the $M$-grading; a smooth
curve $C$; and a divisor $\DD=\sum_{z\in C}\Delta_z\cdot z$ on $C$
whose coefficients $\Delta_z$ are polyhedra in $N_\QQ$ having tail
cone $\sigma$. Furthermore, if $C$ is projective we ask for the
polyhedron $\sum_{z\in C} \Delta_z$ to be a proper subset of
$\sigma$. For every $m\in \sigma^\vee\cap M$ the evaluation of $\DD$
at $m$ is the $\QQ$-divisor given by
$$\DD(m)=\sum_{z\in C}\min_{p\in \Delta_z}p(m)\,.$$
We define the $M$-graded algebra
\begin{align} \label{AH}
A[\DD]=\bigoplus_{m\in \sigma^\vee\cap
  M}A_m\chi^m,\quad\mbox{where}\quad A_m=H^0(C,\OO_C(\DD(m)))\,,
\end{align}
and $\chi^m$ is the corresponding character of the torus $\spec\KK[M]$
seen as a rational function on $\spec A$ via the embedding $\fract
\KK[M]\hookrightarrow \fract A[\DD]=\fract \KK(C)[M]$.

It follows from \cite{AlHa06} that $A[\DD]$ is an normal affine domain
and that every normal affine $M$-graded domain $A$ with $\trdeg
A=\rank M+1$ is equivariantly isomorphic to $A[\DD]$ for some
polyhedral divisor on a smooth curve, see also \cite[Theorem
1.4]{Lie10}.

\medskip

We turn back now to our particular case where we deal with the
polynomial algebra $\KK^{[n]}$ graded by $\chara(\TD)$. First, we need
to fix an isomorphism $M\simeq \chara(\TD)$. Let
$\{\mu_1,\ldots,\mu_{n-1}\}$ be the canonical basis of $M$. We define
the isomorphism $\mu_i\mapsto \chi^{\mu_i}$, $i\in\{1,\ldots,n-1\}$,
where
$$\chi^{\mu_i}:\TD\rightarrow \KK^*,\quad
\gamma=\diag(\gamma_1,\ldots,\gamma_n)\mapsto \gamma_i\,.$$ %
Since for every $\gamma\in \TD$ we have
$\gamma_n=\gamma_1^{-1}\cdots\gamma_{n-1}^{-1}$, the character mapping
$\diag(\gamma_1,\ldots,\gamma_n)\mapsto \gamma_n$ is given by
$\chi^{-\mathds{1}}$, where
$\mathds{1}:=\mu_1+\ldots+\mu_{n-1}$. Under this isomorphism, the
algebra $\KK^{[n]}$ is graded by $M$ via $\deg x_i=\mu_i$, for all
$i\in\{1,\ldots,n-1\}$ and $\deg x_n=-\mathds{1}$.

Let now $\{\nu_1,\ldots,\nu_{n-1}\}$ be a basis of $N$ dual to the
basis $\{\mu_1,\ldots,\mu_{n-1}\}$ of $M$. We also let $\Delta$ be the
standard $(n-1)$-simplex in $N_\QQ$ i.e., the convex hull of the set
$\{\nu_1,\ldots,\nu_{n-1},\bar{0}\}$.

\begin{lemma} \label{isom}
  The $M$-graded algebra $\KK^{[n]}$ is equivariantly isomorphic to
  $A[\DD]$, where $\DD$ is the polyhedral divisor $\DD=\Delta\cdot[0]$
  on $\AF^1$.
\end{lemma}
\begin{proof}
  By \cite{AlHa06}, the $M$-graded algebra $\KK^{[n]}$ is isomorphic
  to $A[\DD]$ for some polyhedral divisor $\DD$ on a smooth curve
  $C$. Since the weight cone $\sigma^\vee$ of $\KK^{[n]}$ is $M_\QQ$,
  the coefficients of $\DD$ are just bounded polyhedra in $N_\QQ$.

  Since $\AF^n$ is a toric variety and the torus $\TD$ is a subtorus
  of the big torus, we can apply the method in \cite[Section
  11]{AlHa06}. In particular, $C$ is a toric curve. Thus $C=\AF^1$ or
  $C=\PP^1$. Furthermore, the graded piece $B_{\bar{0}}\supsetneq
  \KK$ and so $C$ is not projective by \eqref{AH}. Hence $C=\AF^1$.

  The only divisor in $\AF^1$ invariant by the big torus is $[0]$, so
  $\DD=\Delta\cdot[0]$ for some bounded polyhedron $\Delta$ in
  $N_\QQ$. Finally, applying the second equation in \cite[Section
  11]{AlHa06}, a routine computation shows that $\Delta$ can be chosen
  as the standard $(n-1)$-simplex.
\end{proof}

\begin{remark} \label{rem-isom}
  Letting $\AF^1=\spec \KK[t]$, it is easily seen that the isomorphism
  $\KK^{[n]}\simeq A[\DD]$ is given by $x_i=\chi^{\mu_i}$, for all
  $i\in \{1,\ldots,n-1\}$, and $x_n=t\chi^{-\mathds{1}}$.
\end{remark}

In \cite{Lie10} the homogeneous LNDs on an normal affine $M$-graded
domain are classified into 2 types: fiber type and horizontal type. In
the case where the weight cone is $M_\QQ$, there are no LNDs of fiber
type. Thus, $\KK^{[n]}$ only admits homogeneous LNDs of horizontal
type. The homogeneous LNDs of horizontal type are described in
\cite[Theorem 3.28]{Lie10}. In the following, we specialize this
result to the particular case of $A[\DD]\simeq \KK^{[n]}$.

Let $v_i=\nu_i$, $i\in\{1,\ldots,n-1\}$ and $v_n=\bar{0}$, so that
$\{v_1,\ldots,v_n\}$ is the set of vertices of $\Delta$. For every
$\lambda\in \KK^*$, $i\in \{1,\ldots,n\}$, and $e\in M$ we let
$\partial_{\lambda,i,e}:\fract A[\DD]\rightarrow \fract A[\DD]$ be the
derivation given by
$$\partial_{\lambda,i,e}(t^r\cdot\chi^m)=\lambda(r+v_i(m))\cdot
t^{r-v_i(e)-1}\cdot\chi^{m+e},\quad \forall (m,r)\in M\times \ZZ\,.$$
\begin{lemma}[{\cite[Theorem 3.28]{Lie10}}] \label{lnd} %
  If $\partial$ is a non-zero homogeneous LND of $A[\DD]$, then
  $\partial=\partial_{\lambda,i,e}|_{A[\DD]}$ for some $\lambda\in
  \KK^*$, some $i\in \{1,\ldots,n\}$, and some $e\in M$ satisfying
  $v_j(e)\geq v_i(e)+1,\ \forall j\neq i\,.$ Furthermore, $e$ is the
  degree $\deg\partial$.
\end{lemma}

\begin{proof}[Proof of Theorem~\ref{theo}]
  By Lemma~\ref{root-hom} the root vectors of $\KK^{[n]}$ correspond
  to the homogeneous LNDs in the $M$-graded algebra $\KK^{[n]}$. But
  the homogeneous LNDs on $A[\DD]\simeq \KK^{[n]}$ are given in
  Lemma~\ref{lnd}, so we only need to translate the homogeneous LND
  $\partial=\partial_{\lambda,i,e}|_{A[\DD]}$ in Lemma~\ref{lnd} in
  terms of the explicit isomorphism given in Remark~\ref{rem-isom}.

  Let $e=(e_1,\ldots,e_{n-1})$ and $i\in \{1,\ldots,n-1\}$, so that
  $v_i=\mu_i$. The condition $v_j(e)\geq v_i(e)+1$ yields $e_i\leq -1$
  and $e_j\geq e_i+1$, $\forall j\neq i$. Furthermore,
  $\partial(x_k)=\partial(\chi^{\mu_k})=0$, for all $k\neq i$,
  $k\in\{1,\ldots,n-1\}$,
  $\partial(x_n)=\partial(t\chi^{-\mathds{1}})=0$, and
  $$\partial(x_i)=\partial(\chi^{\mu_i})=\lambda
  t^{-e_i-1}\chi^{e+\mu_i}= \lambda
  t^{-e_i-1}\chi^{(e_i+1)\mathds{1}}\chi^{e+\mu_i-
    (e_i+1)\mathds{1}}=\lambda\mathbf{x}^{\boldsymbol{\alpha}}\,,$$
  where $\alpha_i=0$, $\alpha_n=-e_i-1\geq 0$, and
  $\alpha_k=e_k-e_i-1\geq 0$, for all $k\neq i$,
  $k\in\{1,\ldots,n-1\}$. Hence,
  $\partial=\lambda\cdot\mathbf{x}^{\boldsymbol{\alpha}}
  \cdot\tfrac{\partial}{\partial x_i}$, for some $\lambda\in \KK^*$,
  some $i\in \{1,\ldots,n\}$, and some
  $\boldsymbol{\alpha}\in\ZZ_{\geq 0}^n$ such that $\alpha_i=0$.

  Let now $e=(e_1,\ldots,e_{n-1})$ and $i=n$, so that $v_n=0$.  The
  condition $v_j(e)\geq v_n(e)+1$ yields $e_j\geq 1,\, \forall
  j\in\{1,\ldots,n-1\}$. Furthermore,
  $\partial(x_k)=\partial(\chi^{\mu_k})=0$, $k\in\{1,\ldots,n-1\}$, and
  $$\partial(x_n)=\partial(t\chi^{-\mathds{1}})=\lambda
  \chi^{e-\mathds{1}}=\lambda\mathbf{x}^{\boldsymbol{\alpha}}\,,$$
  where $\alpha_n=0$, and $\alpha_k=e_k-1\geq 0$, for all
  $k\in\{1,\ldots,n-1\}$. Hence,
  $\partial=\lambda\cdot\mathbf{x}^{\boldsymbol{\alpha}}
  \cdot\tfrac{\partial}{\partial x_n}$, for some $\lambda\in \KK^*$,
  some $i\in \{1,\ldots,n\}$, and some
  $\boldsymbol{\alpha}\in\ZZ_{\geq 0}^n$ such that $\alpha_n=0$.

  The last assertion of the theorem follows easily from the fact that
  the root corresponding to the homogeneous LND $\partial$ is the
  character $\chi^{\deg\partial}$.
\end{proof}

Finally, we describe the characters that appear as a root of
$\Aut^*(\AF^n)$.

\begin{corollary}
  The character $\chi\in\chara(\TD)$ given by
  $\diag(\gamma_1,\ldots,\gamma_n)\mapsto
  \gamma_1^{\beta_1}\cdots\gamma_n^{\beta_n}$ is a root of
  $\Aut^*(\AF^n)$ with respect to the diagonal torus $\TD$ if and only
  if the minimum of the set $\{\beta_1,\ldots,\beta_n\}$ is achieved
  by one and only one of the $\beta_i$.
\end{corollary}

\begin{proof}
  By Theorem~\ref{theo}, the roots of $\Aut^*(\AF^n)$ are the
  characters $\diag(\gamma_1,\ldots,\gamma_n)\mapsto
  \gamma_1^{\beta_1}\cdots\gamma_n^{\beta_n}$, where $\beta_i=-1$ for
  some $i\in\{1,\ldots,n\}$ and $\beta_j\geq 0$ $\forall j\neq i$. The
  corollary follows from the fact that $\gamma_1\cdots\gamma_n=1$.
\end{proof}

\section*{Acknowledgment}

The author is grateful to Ivan Arzhantsev who pointed out this problem
to us.

\bibliographystyle{alpha}
\bibliography{../math}

\end{document}